%% file: publication.tex
\documentclass[a4paper,reqno]{amsart}

\usepackage{hyperref}
\input{colors.tex}
\input{tikz.tex}
\usepackage{amssymb, mathtools, amsthm, blkarray, bm} 
\usepackage{dsfont} 
\usepackage{stmaryrd} 
\usepackage{thmtools}

\input{abbreviations.tex}

\usepackage[nameinlink,capitalise]{cleveref} 
\crefname{subsection}{Subsection}{Subsections} 
\input{theorems.tex}

\title[Invariance Properties of Davydov-Yetter Cohomology]{Invariance Properties of Davydov-Yetter Cohomology}
\author{Peter Mader}

\date{November 8, 2025}

\begin{document}

    \begin{abstract}
        Davydov-Yetter (DY) cohomology is a cohomology theory for
        linear semigroupal (i.e.~monoidal but not necessarily unital)
        categories and functors, measuring deformations of their coherence
        isomorphisms. We show that DY cohomology is invariant under freely
        adjoining a unit object, and under adjoining colimits. This implies
        that constructions such as $\Ind$-completion and monoidal abelian
        envelope do not change the cohomology.
    \end{abstract}

    \maketitle
    \tableofcontents

    \input{00-intro.tex}
    \input{01-preliminaries.tex}

    \input{02-adjoining-a-unit.tex}
    \input{03-adjoining-limits.tex}

    \input{bibliography.tex}
\end{document}

%% file: colors.tex
\usepackage{color}

\hypersetup{
    colorlinks,
    citecolor=black!60!blue,
    urlcolor=black!60!blue,
    linkcolor=black!60!blue
}

%% file: tikz.tex
\usepackage{pgf,tikz,tikz-cd}
\usetikzlibrary{patterns}
\usetikzlibrary{decorations.markings,decorations.pathreplacing}
\usetikzlibrary{babel}

\tikzset{
    cross/.pic = {
        \draw[rotate = 45] (-#1,0) -- (#1,0);
        \draw[rotate = 45] (0,-#1) -- (0, #1);
    }
}
\tikzset{
    toideal/.style = {draw=\torange,ultra thick}
}
\tikzset{
    toidealthin/.style = {draw=\torange, very thick}
}
\tikzset{
    pics/coordsys/.style 2 args = {
        code = {
            \node[anchor=north east] (0, 0) {\small 1};
            \foreach \x in {1, ..., #1} {
                \draw[thin,black!15] (\x,-.1) -- (\x,#2+.2);
            }
            \foreach \y in {1, ..., #2} {
                \draw[thin,black!15] (-.1, \y) -- (#1+.2, \y);
            }
            \draw[->] (-.1, 0) -- (#1+.2, 0) node[anchor=west] {$x$};
            \draw[->] (0, -.1) -- (0, #2+.2) node[anchor=south] {$y$};
        }
    }
}

\tikzcdset{
    cells={font=\everymath\expandafter{\the\everymath\displaystyle}},
}

\pgfkeys{/tikz/.cd,
    alt double distance/.initial=5pt,
    alt double step/.initial=1pt,
}

\pgfdeclaredecoration{double deco}{initial}
{
    \state{initial}[width=\pgfkeysvalueof{/tikz/alt double step},next state=cont] {
        \pgfmoveto{\pgfpoint{\pgfkeysvalueof{/tikz/alt double step}}{\pgfkeysvalueof{/tikz/alt double distance}/2}}
        \pgfpathlineto{\pgfpoint{0.3\pgflinewidth}{\pgfkeysvalueof{/tikz/alt double distance}/2}}
        \pgfpathmoveto{\pgfpoint{0.3\pgflinewidth}{-\pgfkeysvalueof{/tikz/alt double distance}/2}}
        \pgfpathlineto{\pgfpoint{1pt}{-\pgfkeysvalueof{/tikz/alt double distance}/2}}
        \pgfcoordinate{lastup}{\pgfpoint{1pt}{\pgfkeysvalueof{/tikz/alt double distance}/2}}
        \pgfcoordinate{lastdown}{\pgfpoint{1pt}{-\pgfkeysvalueof{/tikz/alt double distance}/2}}
    }
    \state{cont}[width=\pgfkeysvalueof{/tikz/alt double step}]{
        \pgfmoveto{\pgfpointanchor{lastup}{center}}
        \pgfpathlineto{\pgfpoint{\pgfkeysvalueof{/tikz/alt double step}}{\pgfkeysvalueof{/tikz/alt double distance}/2}}
        \pgfcoordinate{lastup}{\pgfpoint{\pgfkeysvalueof{/tikz/alt double step}}{\pgfkeysvalueof{/tikz/alt double distance}/2}}
        \pgfmoveto{\pgfpointanchor{lastdown}{center}}
        \pgfpathlineto{\pgfpoint{\pgfkeysvalueof{/tikz/alt double step}}{-\pgfkeysvalueof{/tikz/alt double distance}/2}}
        \pgfcoordinate{lastdown}{\pgfpoint{\pgfkeysvalueof{/tikz/alt double step}}{-\pgfkeysvalueof{/tikz/alt double distance}/2}}
    }
    \state{final}[width=0pt]
    { 
        \pgfmoveto{\pgfpointdecoratedpathlast}
    }
}
\tikzset{alt double/.style={decorate,decoration=double deco}}

%% file: abbreviations.tex
\newcommand\cA{\mathcal{A}}
\newcommand\cB{\mathcal{B}}

\newcommand\cC{\mathcal{C}}
\newcommand\cD{\mathcal{D}}

\newcommand\cF{\mathcal{F}}
\newcommand\cI{\mathcal{I}}
\newcommand\cJ{\mathcal{J}}

\newcommand\cP{\mathcal{P}}

\newcommand\cT{\mathcal{T}}

\newcommand\cV{\mathcal{V}}

\newcommand\CVec{\mathrm{Vec}}

\newcommand\CSet{\mathrm{Set}}

\DeclareMathOperator*\colim{colim}

\newcommand\vfi{\varphi}

\newcommand\bGamma{\mathbf{\Gamma}}
\newcommand\bDelta{\mathbf{\Delta}}

\newcommand\oday{\mathrel{\otimes_{\mathrm{Day}}}}
\newcommand\ohat{\mathrel{\hat{\kern0.6pt\otimes}}}

\newcommand\id{\mathrm{id}}
\newcommand\PSh{\mathrm{PSh}}
\newcommand\Sh{\mathrm{Sh}}
\newcommand\Ind{\mathrm{Ind}}

\newcommand\End{\mathrm{End}}

\newcommand\Ob{\mathrm{Ob}}
\newcommand\op{\mathrm{op}}
\newcommand\DY{\mathrm{DY}}

\newcommand\Hom{\mathrm{Hom}}

\newcommand\hoch{\mathrm{hoch}}
\newcommand\Choch{C_\hoch}
\newcommand\dhoch{d_\hoch}

\newcommand\CDY{C_{\mathrm{DY}}}
\newcommand\Hhoch{\mathit{HH}}
\newcommand\HDY{H_{\mathrm{DY}}}

\newcommand\Fun{\mathrm{Fun}}

\newcommand\ex{\mathrm{ex}}
\newcommand\faith{\mathrm{faith}}

\newcommand\uu{\mathrm{u}}
\newcommand\ps{\mathrm{ps}}
\newcommand\sk{\mathrm{sk}}

\newcommand\one{\mathds{1}}
\newcommand{\N}{\ensuremath{\mathbb N}}

\DeclarePairedDelimiter\ceil{\lceil}{\rceil}

%% file: theorems.tex
\newtheoremstyle{style}
{3pt} 
{3pt} 
{} 
{} 
{\bfseries} 
{.} 
{.5em} 
{} 

\theoremstyle{style}

\newtheorem{thm}{Theorem}[section]
\newtheorem{theorem}[thm]{Theorem}
\newtheorem{proposition}[thm]{Proposition}
\newtheorem*{proposition*}{Proposition}
\newtheorem*{theorem*}{Theorem}
\newtheorem*{ack}{Acknowledgements}
\newtheorem{lemma}[thm]{Lemma}
\newtheorem{setup}[thm]{Setup}

\newtheorem{corollary}[thm]{Corollary}

\theoremstyle{definition}
\newtheorem{definition}[thm]{Definition}
\newtheorem{example}[thm]{Example}
\newtheorem{remark}[thm]{Remark}

\newtheorem{thmx}{Theorem}
\newtheorem{propx}{Proposition}

\newtheorem{thmy}{Theorem}

%% file: 00-intro.tex
\section*{Introduction} \label{s:intro}

Linear monoidal categories and functors come with coherence isomorphisms whose
deformations are classified by Davydov-Yetter (DY) cohomology -- much like
deformations of associative algebras are classified by Hochschild cohomology.
As for algebras, the existence of a unit object is irrelevant for deformations
of the associator and for the the cohomology. Consequently, we will often speak
about \textit{semigroupal} categories and functors to emphasize that we do not
require unitality.

For a $k$-linear semigroupal functor $F: \cC \to \cD$ with coherence isomorphism
\[
    \Phi: F(-) \otimes F(-) \xrightarrow{\;\cong\;} F(- \otimes -),
\]
we denote by $H_\DY^\bullet(F)$ its Davydov-Yetter cohomology, with $\Phi$
understood. The vector space $H_\DY^2(F)$ measures infinitesimal deformations
of $F$. If $F = \id_\cC$ is an identity functor, $H_\DY^3(\cC) \coloneqq
H_\DY^3(\id_\cC)$ measures deformations of the associator of~$\cC$
(see~\cite{davydov1997twisting,craneyetter1996,yetter1997}).

Often, cohomology theories are functorial: A morphism $f: A \to B$ between two
objects (groups, sheaves, \dots) canonically induces a morphism $H^\bullet(B)
\to H^\bullet(A)$ on cohomology, which may be used to compare the cohomology
groups of $A$ and $B$. The construction $H^\bullet(-)$ can be seen as a
contravariant functor.

For Hochschild and Davydov-Yetter cohomology, not every morphism $f$ induces a
sensible morphism between cohomology groups. \Cref{s:preliminaries} discusses
for which morphisms this is possible, building on ideals
in~\cite{davydovcenter,Davydov2022DeformationCO}: A morphism between two
$k$-linear semigroupal functors $F: \cC_1 \to \cD_1, G: \cC_2 \to \cD_2$ is a
commutative diagram of semigroupal categories
\begin{equation}
    \begin{tikzcd}
        \cC_1 \arrow[swap]{d}{F} \arrow{r}{S} & \cC_2 \arrow{d}{G} \\
        \cD_1 \arrow[swap]{r}{T} & \cD_2,
    \end{tikzcd}
\end{equation}
with the additional requirement that $T$ be fully faithful on endomorphisms.
Such pairs $(S, T)$ indeed induce a morphism of complexes $C_\DY^\bullet(G) \to
C_\DY^\bullet(F)$ in a functorial way. Specializing $F$ and $G$ to be identity
functors, we obtain the following functoriality result:

\begin{propx}[\Cref{c:functoriality-cat}]
    For each $n \in \N_+$, DY cohomology induces a functor
    \[
        \left\{
            \text{\parbox[m]{5.7cm}{$k$-linear semigroupal categories with
            $k$-linear semigroupal functors fully faithful on endomorphisms}}
        \right\} \longrightarrow \CVec_k^\op.
    \]
\end{propx}

\Cref{s:unit} discusses invariance of DY cohomology under adjoining a formal
unit object to a semigroupal category. As observed by Hochschild
in \cite[\S~2]{hochschild}, Hochschild cohomology is invariant under
unitalization, i.e.~the adjunction of a formal unit element. In order to prove
an analogous statement for DY cohomology, we recall Hochschild's argument for
algebras and then transfer it to the case of semigroupal categories.

A unitalization $\cC^\uu$ of a $k$-linear semigroupal category $\cC$ has an
additional object~$\one$ such that $\Hom_{\cC^\uu}(\one, X) = \{0\}$ and
$\Hom_{\cC^\uu}(X, \one) = \{0\}$ for all $X \in \cC$.

\begin{thmx}[\Cref{p:hdy-unitalization}] \label{thm1}
    Let $\cC^\uu$ be a unitalization of a small $k$-linear semigroupal category
    $\cC$. Assume that $\cC$ is semigroupally equivalent to a strict skeletal
    semigroupal category. Then the morphism $\HDY^n(\cC^\uu) \to \HDY^n(\cC)$
    induced by $\iota: \cC \hookrightarrow \cC^\uu$ is an isomorphism for all
    $n \geqslant 1$.
\end{thmx}

In case all non-zero morphisms in a category are endomorphisms, \Cref{thm1}
implies that the unit object can be ignored when computing the DY cohomology.
This simplification is used tacitly in \cite[\S~3]{Davydov2022DeformationCO}.

\Cref{s:limits} derives a different kind of invariance result: Many completions
or envelopes of linear monoidal categories (additive envelope, pseudo-abelian
envelope, Ind-completion, free cocompletion) do not
change the DY cohomology. More precisely:

\begin{thmx}[\Cref{t:colimit-invariant}]
    Consider a commutative diagram of $k$-linear semigroupal categories
    \[\begin{tikzcd}
        \cC_0 \arrow{r}{F_0} \arrow[hook]{d} \arrow[bend right=50,swap]{dd}{Y_{\cC_0}} &[1.3cm] \cD_0 \arrow[hook,swap]{d} \arrow[bend left=50]{dd}{Y_{\cD_0}} \\
        \cC \arrow{r}{F} \arrow[hook]{d} & \cD \arrow[hook,swap]{d} \\
        \PSh(\cC_0) \arrow{r}{\PSh(F_0)} & \PSh(\cD_0),
    \end{tikzcd}\]
    where $Y$ denotes the Yoneda embedding and $\hookrightarrow$ denotes fully
    faithful functors. Then we have an isomorphism of complexes
    $C_\DY^\bullet(F) \to C_\DY^\bullet(F_0)$.
\end{thmx}

In particular, $k$-linear semigroupal subcategories $\cC_0 \subseteq \cC
\subseteq \PSh(\cC_0)$ all have isomorphic cohomologies and therefore
corresponding deformations. For example, if~$\cC$ is abelian with enough
projectives, then $\cC_0$ may be taken to be the full subcategory of projective
objects~(\Cref{x:enough-projectives}). In~\Cref{t:abelian-envelope}, we deduce
that DY cohomology is invariant under those monoidal abelian envelopes of
pseudo-tensor categories which arise from the construction given by Coulembier
in~\cite{Cou}.

\begin{ack}
    The results presented here first appeared in my master's thesis \cite{ma}
    supervised by Catharina Stroppel. I heartily thank her and Johannes Flake
    for their guidance, helpful comments, and feedback during the development
    of these results and the preparation of this paper.
\end{ack}

%% file: 01-preliminaries.tex
\section{Preliminaries} \label{s:preliminaries}

First we recall some notions which we frequently use. The term
\textit{semigroupal category} (i.e.~``monoidal category without a unit'') was
taken from \cite{semigroupal}.

\begin{definition}
    A \textbf{semigroupal category} is a category $\cC$ equipped with a
    bifunctor $\otimes: \cC \times \cC \to \cC$ (the \textbf{semigroupal product})
    and a natural isomorphism
    \[
        a_{XYZ}: X \otimes (Y \otimes Z) \xrightarrow{\;\cong\;} (X \otimes Y)
        \otimes Z, \qquad X, Y, Z \in \cC
    \]
    called the \textbf{associator}, satisfying the pentagon axiom. It is called \textbf{monoidal} if there exists
    an object $\one \in \cC$ and an isomorphism $i: \one \otimes \one \to \one$
    such that the endofunctors $- \otimes \one$ and $\one \otimes -$ are
    equivalences.
\end{definition}

For a detailed account on the pentagon axiom and monoidal categories, see
\cite[\S~2.1]{EGNO}. The unit $(\one, i)$ of a monoidal category is essentially
unique (see \cite[Proposition~2.2.6]{EGNO}).

\begin{definition}
    A \textbf{semigroupal functor} between two semigroupal categories
    $(\cC, \otimes, a)$ and $(\cD, \otimes, b)$ is a functor $F: \cC \to
    \cD$ equipped with a natural isomorphism
    \[
        \Phi_{XY}: F(X) \otimes F(Y) \xrightarrow{\;\cong\;} F(X \otimes Y),
        \qquad X, Y \in \cC
    \]
    satisfying the hexagon axiom (see \cite[\S~2.4]{EGNO}). If $\cC, \cD$ are
    monoidal with unit objects $\one_\cC, \one_\cD$, respectively, then $(F,
    \Phi)$ is called \textbf{monoidal} if $F(\one_\cC) \cong \one_\cD$.
\end{definition}

\begin{definition}
    Let $k$ be a field. A semigroupal category $(\cC, \otimes, a)$ is called
    \textbf{\mbox{$k$-linear} semigroupal} if the category $\cC$ and the functors $(X
    \otimes -), (- \otimes X)$ for each $X \in \cC$ are $k$-linear. A
    semigroupal functor $(F, \Phi)$ is $k$-linear if $F$ is.

    By a \textbf{tensor category} over $k$, we mean an essentially small
    abelian rigid symmetric $k$-linear monoidal category with $\End(\one)
    \cong k$. A \textbf{pseudo-tensor category} is merely pseudo-abelian
    (i.e. additive and idempotent-complete) instead of abelian.
\end{definition}

\subsection{Davydov-Yetter cohomology}

Let $k$ be field and let $F: \cC \to \cD$ be a strict $k$-linear semigroupal
functor between strict $k$-linear semigroupal categories. For~$n \in \N_+$
consider the functor
\[
    F^{\otimes,n}: \cC^{\times n} \longrightarrow \cD, \qquad (X_1, \dots, X_n)
    \longmapsto F(X_1 \otimes \cdots \otimes X_n)
\]
defined on morphisms in the obvious way. An endomorphism of
$F^{\otimes,n}$ is a collection
\[
    \eta_{X_1, \dots, X_n}: F(X_1 \otimes \dots \otimes X_n) \longrightarrow
    F(X_1 \otimes \cdots \otimes X_n), \qquad X_1, \dots, X_n \in \cC
\]
of endomorphisms in $\cD$ which is natural in each of the $X_i$. For $n \in
\N_+$, define the vector space $C_\DY^n(F) \coloneqq \End(F^{\otimes,n})$ and,
for each $i \in \{0, \dots, n+1\}$, the linear map $\partial_i^n: C_\DY^n(F)
\to C_\DY^{n+1}(F)$ given by
\begin{equation} \label{e:coface-dy}
    \partial^n_i(\eta)_{X_0,\dots,X_n} = \begin{cases}
        \id_{F(X_0)} \otimes \eta_{X_1,\dots,X_n} & i = 0 \\
        \eta_{X_1,\dots,X_{i-2},X_{i-1}\otimes X_i,X_{i+1},\dots,X_n} & 1
        \leqslant i \leqslant n \\
        \eta_{X_0,\dots,X_{n-1}} \otimes \id_{F(X_n)} & i = n+1.
    \end{cases}
\end{equation}
These maps satisfy the coface relations
\begin{equation}\label{coface-relations}
    \partial^{n+1}_j\partial^n_i = \partial^{n+1}_i\partial^n_{j-1}, \qquad 0
    \leqslant i < j \leqslant n+2,
\end{equation}
so the collection $C_\DY^\bullet(F)$ is a cochain complex with differentials
\[
    d_\DY^n \coloneqq \sum_{i=0}^{n+1} (-1)^i \partial^n_i\;:\;C_\DY^n(F)
    \xrightarrow{\hphantom{1cm}} C_\DY^{n+1}(F).
\]

\begin{definition} \label{d:cdy}
    The cochain complex $C_\DY^\bullet(F)$ is called the \textbf{Davydov-Yetter
    (cochain) complex of $F$}. Its cohomology is called the
    \textbf{Davydov-Yetter cohomology of $F$} and is denoted by
    $\HDY^\bullet(F)$. In the special case $F = \id_\cC$, we also use the
    notations $\CDY^\bullet(\cC)$ and $\HDY^\bullet(\cC)$, respectively.
\end{definition}

\begin{remark} \label{r:non-strict}
    Davydov-Yetter cohomology as defined by Crane and Yetter
    \cite{craneyetter1996,yetter1997} works just as well for non-strict
    semigroupal categories and functors. We briefly explain how \Cref{d:cdy}
    generalizes to non-strict \mbox{$(\cC, \otimes, a_\cC), (\cD, \otimes,
    a_\cD)$} and $(F, \Phi): \cC \to \cD$.

    Consider first the following example: For objects $X_1, X_2, X_3 \in \cC$,
    there are multiple expressions involving semigroupal products, parentheses
    and the functor~$F$, e.g.~$A \coloneqq F(X_1 \otimes X_2) \otimes F(X_3), B
    \coloneqq (F(X_1) \otimes F(X_2)) \otimes F(X_3)$. Given any morphism $f: A
    \to B$ in $\cC$, we can pre- and postcompose with the associators $a_\cC,
    a_\cD$ and structure isomorphism $\Phi$ and their inverses to obtain the
    morphism $\ceil{f}$ given by
    \[\begin{tikzcd}
        F((X_1 \otimes X_2) \otimes X_3) \arrow[swap]{d}{\Phi^{-1}} & F(X_1) \otimes (F(X_2) \otimes F(X_3))). \\
        F(X_1 \otimes X_2) \otimes F(X_3) \arrow{r}{f} & (F(X_1) \otimes
        F(X_2)) \otimes F(X_3) \arrow[swap]{u}{a_\cD}
    \end{tikzcd}\]
    There are other combinations of $a_\cC, a_\cD, \Phi$ that can be applied to $f$
    to produce a morphism of the same signature as $\ceil{f}$. All such
    combinations coincide by MacLane's coherence theorem
    (\cite[Theorem~2.9.2]{EGNO}; for an elementary proof, see
    \cite[\S~XI.5]{kasselqg}).

    More generally, given $X_1, \dots, X_n \in \cC$, two parenthesizations $A,
    B \in \cD$ of the expression $F(X_1) \otimes \cdots \otimes F(X_n)$ and a
    morphism $f: A \to B$, the \textit{padding} operator~$\ceil{-}$ assigns to
    $f$ the morphism $\ceil{f}$ with signature
    \begin{equation} \label{e:signature-padding}
        F(\underbrace{(\cdots(X_1 \otimes X_2) \otimes
        \cdots) \otimes X_n}_{\text{left parenthesized}}) \to
        \underbrace{F(X_1) \otimes (\cdots \otimes (F(X_{n-1}) \otimes F(X_n))
        \cdots)}_{\text{right parenthesized}}
    \end{equation}
    obtained by applying appropriate associators of $\cC, \cD$ and structure
    isomorphisms~$\Phi$. Now define $\CDY^\bullet(F)$ to be the space of natural
    transformations with signature~\eqref{e:signature-padding}
    and replace the coface maps $\partial_i^n$ defined in~\eqref{e:coface-dy}
    by their padded versions. Again, this yields a complex called the
    \textit{Davydov-Yetter cochain complex} of $F$. If $\cC, \cD, F$ happen to
    be strict, then this construction clearly recovers~\Cref{d:cdy}. Moreover,
    it is easy to see that this complex is preserved under semigroupal
    equivalence and in particular under strictification; this allows us to
    restrict to the strict case in most situations.
\end{remark}

Note that~\Cref{d:cdy} defines $\CDY^n(F)$ and $d_\DY^n$ only in positive
degrees, as do \cite{craneyetter1996,yetter1997}. To extend it to $n=0$, one
should assume that $\cC, \cD, F$ are not only semigroupal but monoidal. Under
this assumption, \cite{EGNO} defines $\cC^{\times 0}$ to be the category with a single
object $*$ and a single morphism $\id_*$. Then $F^{\otimes,0}: \cC^{\times 0}
\to \cD$ is the functor sending $* \mapsto \one_\cD$ and we obtain
\[
    \CDY^0(F) \coloneqq \End(F^{\otimes,0}) \cong \End(\one_\cD)
\]
and
\[
    d_\DY^0: \End(\one_\cD) \to \End(F), \qquad d_\DY^0(f)_X = \id_{F(X)}
    \otimes f - f \otimes \id_{F(X)}.
\]

\subsection{(Non-)Functoriality} \label{ss:functoriality}

Unlike many cohomology theories, Davydov-Yetter cohomology does not immediately
give rise to a functor
\[
    H_\DY^n: \left\{
        \text{\parbox[m]{5.4cm}{$k$-linear semigroupal categories with
        $k$-linear semigroupal functors}}
    \right\} \longrightarrow \CVec_k^\op
\]
because it is not clear how a semigroupal functor induces a map on cohomology.
As noted by Davydov, this is parallel to centers and the Hochschild cohomology
of associative algebras: There is no functor $\text{$k$-Alg} \to
\text{$k$-CAlg}$ sending an algebra to its center. However, the center can seen
as a functor if we modify the codomain category: A morphism of algebras $\vfi:
A \to B$ induces a cospan of algebras
\[\begin{tikzcd}
    Z(A) \arrow{r}{\vfi} & Z_B(\vfi(A)) & Z(B) \arrow[hook']{l}
\end{tikzcd}\]
and $Z(-)$ extends to a functor from $\text{$k$-Alg}$ to an appropriate
category of cospans (see~\cite[\S~4]{davydovcenter}). Moreover, if $\vfi$ is
surjective, then $Z(B) = Z_B(\vfi(A))$ and we indeed obtain a morphism $Z(A)
\to Z(B)$ between the centers.

Similarly, a $k$-linear semigroupal functor $F: \cC \to \cD$ induces a cospan
of complexes $C_\DY^\bullet(\cC) \rightarrow C_\DY^\bullet(F) \leftarrow
C_\DY^\bullet(\cD)$; the first arrow is an isomorphism if $F$ is fully faithful
on endomorphisms (see \cite[\S~2.1]{Davydov2022DeformationCO}). We use these
ideas to formulate a kind of functoriality for $C_\DY^\bullet$ both for
semigroupal categories and semigroupal functors.

%

Consider a commutative square of $k$-linear semigroupal categories and functors
\begin{equation} \label{e:functoriality-square} \tag{$\square$}
    \begin{tikzcd}
        \cC_1 \arrow[swap]{d}{F} \arrow{r}{S} & \cC_2 \arrow{d}{G} \\
        \cD_1 \arrow[swap]{r}{T} & \cD_2,
    \end{tikzcd}
\end{equation}
where $T$ is fully faithful on endomorphisms. This diagram $\square$ induces a
morphism of complexes $\vfi_\square: \CDY^\bullet(G) \to \CDY^\bullet(F)$ in
the following way. We denote the inverses of the bijections $\End_{\cD_1}(X)
\to \End_{\cD_2}(T(X))$ simply by $T^{-1}$. It is easy to check that for $n \in
\N_+$, the linear maps
\begin{equation*} \label{e:functoriality-morphism} \begin{aligned}
    &\CDY^n(G) \longrightarrow \CDY^n(F), \quad &&\eta \longmapsto
    (T^{-1}\eta_{S(X_1),\dots,S(X_n)})_{X_1,\dots,X_n \in \cC_1}, \\
\end{aligned}\end{equation*}
are well-defined and constitute a morphism of complexes $\vfi_\square$. This
construction makes $\CDY^\bullet(-)$ functorial in the following sense: Let
$\cF$ be the category with objects the $k$-linear semigroupal functors and with
morphisms between $F, G$ the squares as in~\eqref{e:functoriality-square}. Composition is
given by placing squares side by side. Then the above implies the
following statement:

\begin{proposition} \label{p:functoriality-square}
    For each $n \in \N_+$, DY cohomology induces a functor $\cF \to \CVec_k^\op$
    sending $F$ to $C_\DY^n(F)$ and a square~\eqref{e:functoriality-square} to
    the morphism~$\vfi_\square$.
\end{proposition}

The following special cases arise from restricting $\cF$ to a subcategory where
$T$ or where $F$ and $G$ are the identity functors.

\begin{corollary}
    Fix a $k$-linear semigroupal category $\cD$ and let
    $\text{$k$-\text{SemigrpCat}} / \cD$ be the overcategory with objects the
    $k$-linear semigroupal functors $\cC \to \cD$ and with morphisms
    the commutative triangles of the form
    \[\begin{tikzcd}
        \cC_1 \arrow{rr}{S} \arrow[swap]{rd}{F} && \cC_2 \arrow{ld}{G} \\
        & \cD &
    \end{tikzcd}\]
    For each $n \in \N_+$, DY cohomology induces a functor
    $\text{$k$-\text{SemigrpCat}} / \cD \to \CVec_k^\op$.
\end{corollary}

\begin{corollary} \label{c:functoriality-cat}
    For each $n \in \N_+$, DY cohomology induces a functor
    \[
        H_\DY^n: \left\{
            \text{\parbox[m]{5.7cm}{$k$-linear semigroupal categories with
            $k$-linear semigroupal functors fully faithful on endomorphisms}}
        \right\} \longrightarrow \CVec_k^\op.
    \]
\end{corollary}

\begin{example} \label{c:forgetful}
    Let $\cC \subseteq \cD$ be a full $k$-linear semigroupal subcategory. Then
    the forgetful morphism of complexes $C_\DY^\bullet(\cD) \to
    C_\DY^\bullet(\cC)$ is induced by the inclusion $\cC \hookrightarrow \cD$
    via \Cref{c:functoriality-cat}.
\end{example}

%% file: 02-adjoining-a-unit.tex
\section{Adjoining a unit} \label{s:unit}

DY cohomology is fundamentally a cohomology theory for (non-unital) semigroupal
categories, just as Hochschild cohomology is a cohomology theory for non-unital
algebras. As observed by Hochschild \cite[\S~2]{hochschild}, Hochschild
cohomology is invariant under unitalization, i.e.~the adjunction of a formal
unit element. In order to prove an analogous statement for DY cohomology, we
recall the central points of Hochschild's argument for algebras and later
transfer it to the case of semigroupal categories.

\subsection{Unitalization of algebras} \label{ss:unitalization-alg}

Let $k$ be a field, $A$ a $k$-algebra and $M$ an $(A, A)$-bimodule. Even if
there is a unit $1 \in A$, we do not require $1 \cdot m = m = m \cdot 1$ for~$m
\in M$.

\begin{definition}
    The \textbf{unitalization} $A^\uu$ of the $k$-algebra $A$ is the $k$-algebra
    with underlying $k$-vector space $A \oplus k$ and multiplication $(a,
    \lambda)(b, \mu) \coloneqq (ab + \lambda b + \mu a, \lambda\mu)$ for all
    $a, b \in A$ and scalars $\lambda, \mu \in k$.
\end{definition}

Clearly, the unitalization is unital with unit $(0, 1)$ and $A$ is canonically
a subalgebra of $A^\uu$ via $\iota: A \hookrightarrow A^\uu, a \mapsto (a, 0)$.
In addition, $M$ becomes an $(A^\uu, A^\uu)$-bimodule by letting the new unit
element $(0, 1)$ act as the identity.

The vector space $\Choch^1(A, M) = \Hom_k(A, M)$ becomes an $(A, A)$-bimodule
by defining
\begin{equation} \label{e:action-on-hom}
    (a \cdot f)(b) \coloneqq a \cdot f(b), \qquad (f \cdot a)(b) \coloneqq
    f(ab) - f(a) \cdot b
\end{equation}
for all $a, b \in A$ and $f \in \Hom_k(A, M)$. For $n \in \N_+$, we have the
currying and uncurrying isomorphisms
\begin{equation} \label{e:currying}\begin{aligned}
    \Gamma: \Choch^n(A, M) &\longrightarrow \Choch^{n-1}(A, \Hom_k(A, M)) \\
    f & \longmapsto \left((a_1, \dots, a_{n-1}) \mapsto (a_n \mapsto
    f(a_1, \dots, a_n))\right) \\
    \left((a_1, \dots, a_n) \mapsto f(a_1, \dots, a_{n-1})(a_n)\right) & \longmapsfrom f.
\end{aligned}\end{equation}
Currying commutes with differentials, i.e.~the isomorphisms~\eqref{e:currying}
yield an isomorphism of complexes $\Choch^n(A, M) \cong \Choch^{n-1}(A,
\Hom_k(A, M))$. This allows us to perform induction on the degree of Hochschild
cochains and obtain the following crucial ingredient (see
{\cite[Lemma~1]{hochschild}}):

\begin{lemma}
    \label{l:cohomologous-to-1}
    Assume that the $k$-algebra $A$ has a unit $1 \in A$ and let $f \in
    \Choch^n(A, M)$ be an $n$-cocycle for some $n \in \N_+$. Then
    there exists a cochain $h \in \Choch^{n-1}(A, M)$ such that $(f +
    \dhoch^{n-1}h)(a_1, \dots, a_n) = 0$ whenever $a_i = 1$ for some $i$.
\end{lemma}
\begin{proof}
    We proceed by induction on $n$. For $n = 1$, take $h \coloneqq 1 \cdot f(1)
    - f(1) \cdot 1$. Then
    \[\begin{aligned}
        (f - \dhoch^0h)(a) &= f(a) - a \cdot h + h \cdot a \\
        &= f(a) - a \cdot f(1) + a \cdot f(1) \cdot 1 + 1 \cdot f(1) \cdot a -
        f(1) \cdot a.
    \end{aligned}\]
    Since $f$ is a 1-cocycle, we have $f(1) = 1 \cdot f(1) + f(1) \cdot 1$, so
    $1 \cdot f(1) = 1 \cdot f(1) + 1 \cdot f(1) \cdot 1$, which yields $1 \cdot
    f(1) \cdot 1 = 0$. Therefore, substituting $a = 1$, the above equation
    becomes $(f-\dhoch^0h)(1) = 0$.

    For the induction step, assume that $n > 1$. Since currying commutes with
    differentials, $\Gamma(f) \in \Choch^{n-1}(A, \Hom_k(A, M))$ is an
    $(n-1)$-cocycle. The induction hypothesis yields a cochain $p \in
    \Choch^{n-2}(A, \Hom_k(A, M))$ such that $(\Gamma(f) + \dhoch^{n-2}p)(a_1,
    \dots, a_{n-1}) = 0$ whenever $a_i = 1$ for some $i$. Consequently, $g
    \coloneqq \Gamma^{-1}(p) \in \Choch^{n-1}(A, M)$ satisfies
    \[
        f'(a_1, \dots, a_n) \coloneqq (f + \dhoch^{n-1}g)(a_1, \dots,
        a_n) = 0
    \]
    whenever $a_i = 1$ for some $i \in \{1, \dots, n-1\}$. To make the
    expression vanish also for $a_n = 1$, we need to modify $g$ in the
    following way: Consider the $(n-1)$-cochain $g'$ given by
    \[
        g'(a_1, \dots, a_{n-1}) \coloneqq f'(a_1, \dots, a_{n-1}, 1) - 2f'(a_1,
        \dots, a_{n-1}, 1) \cdot 1,
    \]
    where the $2$ denotes scalar multiplication. Using the fact that $f'$ is a
    cocycle, a direct calculation shows
    \[\begin{aligned}
        &\;(\dhoch^{n-1}g')(a_1, \dots, a_n) \\
        =&\;(-1)^n \big(-f'(a_1, \dots, a_{n-1}, 1) \cdot a_n - f'(a_1, \dots,
        a_n) + f'(a_1, \dots, a_n) \cdot 1 \big).
    \end{aligned}\]
    Now choose $h = g + (-1)^n g'$. We obtain
    \[\begin{aligned}
        (f + \dhoch^{n-1}h)(a_1, \dots, a_n) =&\;f'(a_1, \dots, a_n) +
        (-1)^n (\dhoch^{n-1}g')(a_1, \dots, a_n) \\
        =&\;f'(a_1, \dots, a_n) - f'(a_1, \dots, a_{n-1}, 1) \cdot a_n \\
        &- f'(a_1, \dots, a_n) + f'(a_1, \dots, a_n) \cdot 1 \\
        =&\;f'(a_1, \dots, a_n) \cdot 1 - f'(a_1, \dots, a_{n-1}, 1) \cdot a_n.
    \end{aligned}\]
    This expression clearly vanishes for $a_n = 1$, and it vanishes for $a_i =
    0$ with $i \in \{1, \dots, n-1\}$ by construction of $f'$.
\end{proof}

Recall that we have a morphism of algebras $\iota: A \hookrightarrow A^\uu, a
\mapsto (a, 0)$. The projection $\pi: A^\uu \twoheadrightarrow A, (a, \lambda)
\mapsto a$ is not in general a morphism of algebras, but still induces maps
\[\begin{aligned}
    \pi^*: \Choch^n(A, M) &\longrightarrow \Choch^n(A^\uu, M) \\
    f &\longmapsto \left((a_1, \dots, a_n) \mapsto f(\pi(a_1), \dots,
    \pi(a_n))\right)
\end{aligned}\]
which are readily seen to constitute a morphism of complexes. From this setup,
Hochschild deduces invariance under adjoining a unit (see {\cite[Theorem
2]{hochschild}}):

\begin{proposition}
    \label{p:unitalization-alg} The morphism of complexes $\iota^*:
    \Choch^\bullet(A^\uu, M) \to \Choch^\bullet(A, M)$ induced by $\iota: A
    \hookrightarrow A^\uu$ descends to an isomorphism $\Hhoch^n(A^\uu, M) \cong
    \Hhoch^n(A, M)$ for all $n \geqslant 1$.
\end{proposition}
\begin{proof}
    Let $n \in \N_+$. We show that $\iota^*: \Hhoch^n(A^\uu, M) \to \Hhoch^n(A,
    M)$ is bijective. For injectivity, let $[f] \in \ker(\iota^*)$, i.e.~$f \in
    \Choch^n(A^\uu, M)$ is an $n$-cocycle such that $f(a_1, \dots, a_n) = 0$
    for all $a_1, \dots, a_n \in A$. Since $A^\uu$ is unital, we may apply
    \Cref{l:cohomologous-to-1} to obtain a cocycle $g \in \Choch^n(A^\uu, M)$
    such that $[f] = [g]$ and $g(a_1, \dots, a_n) = 0$ whenever $a_i = 1 \in
    A^\uu$ for some $i$. Since $g$ is linear in every argument, this implies
    that $g = 0$, so $[f] = [g] = 0$.

    For surjectivity, let $f \in \Choch^n(A, M)$ be an $n$-cocycle. Then
    $\pi^*f \in \Choch^n(A^\uu, M)$ is also an $n$-cocycle and $\iota^*\pi^*f =
    (\pi\iota)^*f = f$, so $\iota^*([\pi^*f]) = [f]$.
\end{proof}

\Cref{l:cohomologous-to-1} is central to the proof of
\Cref{p:unitalization-alg}. It will later be used in a slightly stronger form,
namely for a subcomplex of $\Choch^\bullet(A, M)$ satisfying certain stability
conditions:

\begin{definition} \label{d:stable-ev-mul}
    We say that a subcomplex $D^\bullet \subseteq \Choch^\bullet(A, M)$ is
    \textbf{stable under evaluation and multiplication} if it meets the
    following stability condition: If $f \in D^n$, $b \in A$ and $i \in \{1,
    \dots, n\}$, then the following function is in $D^{n-1}$:
    \begin{equation} \label{e:cond-1}
        (a_1, \dots, a_{n-1}) \longmapsto f(a_1, \dots, a_{i-1}, b, a_i, \dots,
        a_{n-1})
    \end{equation}
    and the following functions are in $D^{n+1}$:
    \begin{equation} \label{e:cond-2}
        (a_1, \dots, a_{n+1}) \longmapsto f(a_1, \dots, a_{i-1},
        a_ia_{i+1}, a_{i+2}, \dots, a_n),
    \end{equation}
    \begin{equation} \label{e:cond-3}
        (a_1, \dots, a_{n+1}) \mapsto a_1 f(a_2, \dots, a_{n+1}),
    \end{equation}
    \begin{equation} \label{e:cond-4}
        (a_1, \dots, a_{n+1}) \mapsto f(a_1, \dots, a_n) a_{n+1}.
    \end{equation}
\end{definition}

\begin{lemma} \label{l:stable-cohomologous-to-1}
    Assume that $A$ is unital and let $D^\bullet \subseteq \Choch^\bullet(A,
    M)$ be a subcomplex stable under evaluation and multiplication. Let $f \in
    D^n$ be an $n$-cocycle for some $n \in \N_+$. Then there exists $h \in
    D^{n-1}$ such that $(f+\dhoch^{n-1}h)(a_1, \dots, a_n) = 0$ whenever $a_i =
    1$ for some $i$.
\end{lemma}
\begin{proof}
    Note that by definition of the $(A, A)$-bimodule structure on $\Hom_k(A,
    M)$ in~\eqref{e:action-on-hom}, the curried complex $\Gamma(D^\bullet)
    \subseteq \Choch^\bullet(A, \Hom_k(A, M))$ is also stable under evaluation
    and multiplication. Consequently, the cochain $h$ in the proof of
    \Cref{l:cohomologous-to-1} is constructed as a linear combination of
    functions that arise from repeated application of the
    functions~\eqref{e:cond-1}--\eqref{e:cond-4}.
\end{proof}

\subsection{Unitalization of semigroupal categories} \label{ss:unitalization-cat}

For ``sufficiently'' strict categories, the DY complex may be seen as a
subcomplex of some Hochschild complex. This allows us to transport the results
of the previous subsection to the case of semigroupal categories. Here is the
analog of the unitalization $A^\uu$ of an algebra $A$:

\begin{definition} \label{d:unitalization-cat}
    Let $\cC$ be a $k$-linear semigroupal category. A $k$-linear monoidal
    category $\cC^\uu$ with unit object $\one$ is called a
    \textbf{unitalization} of $\cC$ if $\cC \subseteq \cC^\uu$ is the full
    semigroupal subcategory of objects not isomorphic to $\one$ and
    \[
        \Hom_{\cC^\uu}(\one, X) = \{0\}, \qquad \Hom_{\cC^\uu}(X, \one) =
        \{0\}, \qquad X \in \cC.
    \]
\end{definition}

In other words, a unitalization $\cC^\uu$ is obtained from $\cC$ by freely
adjoining a unit object~$\one$. It is unique up to choice of the endomorphism
algebra $\End_{\cC^\uu}(\one)$. Recall that the category $\cC$ is
\textit{skeletal} if $X \cong Y$ implies $X = Y$ for all $X, Y \in \cC$. Below,
we will prove the following:

\begin{theorem} \label{p:hdy-unitalization}
    Let $\cC^\uu$ be a unitalization of a small $k$-linear semigroupal category
    $\cC$. Assume that $\cC$ is semigroupally equivalent to a skeletal strict
    semigroupal category. Then the morphism $\HDY^n(\cC^\uu) \to \HDY^n(\cC)$
    induced by $\iota: \cC \hookrightarrow \cC^\uu$ is an isomorphism for all
    $n \geqslant 1$.
\end{theorem}

\begin{remark}
    The requirement of a strict semigroupal skeleton is not always satisfied:
    Some semigroupal or monoidal categories do not admit a skeletal
    strictification (see e.g.~\cite[Remark~2.8.7]{EGNO}).

    Any category $\cC$ is equivalent to a skeletal category $\cC_\sk$, its
    ``skeleton''. If~$\cC$ carries additional structure such as being
    semigroupal or $k$-linear, then by transfer of structure the category
    $\cC_\sk$ and the equivalence $\cC \to \cC_\sk$ can also be made
    semigroupal or $k$-linear. However, the resulting semigroupal category
    $\cC_\sk$ may be non-strict even if $\cC$ is.
\end{remark}

For the rest of this subsection, let $\cC$ be a small skeletal strict
$k$-linear semigroupal category. Then the objects $S = \Ob(\cC)$ of $\cC$ form
a semigroup under $\otimes$ and we may consider the semigroup algebra $k[S]$:
It is the $k$-vector space with basis $S$ and multiplication extended linearly
from $\otimes$. The $k$-vector space $M = \prod_{X \in S}
\End_\cC(X)$ becomes a $(k[S], k[S])$-bimodule by defining, for all $X, Y \in
S$ and $f \in \End_\cC(Y)$,
\[
    X \cdot f \coloneqq \id_X \otimes f \in \End_\cC(X \otimes Y), \qquad f
    \cdot X \coloneqq f \otimes \id_X \in \End_\cC(Y \otimes X).
\]

\begin{lemma} \label{l:cdy-stable}
    The inclusions
    \[\begin{aligned}
        \CDY^n(\cC) &\lhook\joinrel\longrightarrow \Choch^n(k[S], M), \\
        \eta &\longmapsto \left((X_1, \dots, X_n) \mapsto \eta_{X_1,\dots,X_n}
        \right)
    \end{aligned}\]
    for $n \in \N_+$ realize $\CDY^\bullet(\cC)$ as a subcomplex of
    $\Choch^\bullet(k[S], M)$ that is stable under evaluation and
    multiplication in the sense of \Cref{d:stable-ev-mul}.
\end{lemma}
\begin{proof}
    The algebra structure on $k[S]$ and the bimodule structure on $M$ are
    defined in such a way that the inclusions preserve the respective
    differentials and therefore assemble into a morphism of complexes. For
    $\eta \in \CDY^n(\cC)$ and $X_i \in \cC$ and a fixed $Y \in \cC$, the
    morphisms
    \[\begin{aligned}
        \eta_{X_1, \dots, X_{i-1}, Y, X_i, \dots, X_{n-1}}, &\qquad \eta_{X_1,
        \dots, X_i \otimes X_{i+1}, \dots, X_n}, \\
        \id_{X_1} \otimes \eta_{X_2, \dots, X_{n+1}}, &\qquad \eta_{X_1, \dots,
        X_n} \otimes \id_{X_{n+1}}
    \end{aligned}\]
    are natural in the $X_i$, which implies stability under evaluation and
    multiplication.
\end{proof}

\begin{proof}[Proof of~\Cref{p:hdy-unitalization}]
    Since unitalization and Davydov-Yetter cohomology are invariant under
    semigroupal equivalence, we may assume $\cC$ to be strict and skeletal and
    invoke \Cref{l:cdy-stable}.

    To see that $\iota^*: \HDY^n(\cC^\uu) \to \HDY^n(\cC)$ is injective, let
    $[\eta] \in \ker(\iota^*)$, i.e.~$\eta_{X_1,\dots,X_n} = 0$ for all $X_1,
    \dots, X_n \in S$. Since $\cC^\uu$ is unital, \Cref{l:cdy-stable} implies
    that $\CDY^\bullet(\cC^\uu)$ satisfies the assumptions of
    \Cref{l:stable-cohomologous-to-1}. Thus there exists $\zeta \in
    \CDY^n(\cC^\uu)$ such that $[\zeta] = [\eta] \in \HDY^n(\cC^\uu)$ and
    $\zeta_{X_1,\dots,X_n} = 0$ whenever $X_i = \one$ for some $i$. Altogether,
    we have $g = 0$ and therefore $[f] = 0$, as desired.

    To see that $\iota^*$ is surjective, let $\eta \in \CDY^n(\cC)$. We can
    extend $\eta$ to a cochain $\eta' \in \CDY^n(\cC^\uu)$ by setting
    \[
        \eta'_{X_1, \dots, X_n} = \begin{cases}
            \eta_{X_1, \dots, X_n} & X_1, \dots, X_n \in S \\
            0 & X_i = \one \text{ for some } i.
        \end{cases}
    \]
    for all $X_1, \dots, X_n \in \cC^\uu$. This yields a natural transformation
    due to naturality of $\eta$ and the requirements $\Hom_{\cC^\uu}(X, \one) =
    \Hom_{\cC^\uu}(\one, X) = 0$ for all $X \in S$. Therefore, $\eta'$ is a
    well-defined preimage of $\eta$ under $\iota^*$.
\end{proof}

%% file: 03-adjoining-limits.tex
\section{Adjoining Colimits} \label{s:limits}

Linear monoidal categories are often constructed from smaller categories $\cC$
using various types of completion. Frequent examples include the additive envelope $\cC_\oplus$
(adjunction of direct sums), pseudo-abelian or Karoubi envelope $\cC_\ps$
(adjunction of direct summands), Ind-completion $\Ind(\cC)$ (adjunction of
filtered colimits), cocompletion $\PSh(\cC)$, and more. They
often satisfy a universal property, implying that a linear monoidal functor $F:
\cC \to \cD$ can be upgraded uniquely to an (additive, cocomplete, ...) functor
$F_\oplus: \cC_\oplus \to \cD_\oplus, \Ind(F): \Ind(\cC) \to \Ind(\cD)$.

In this section, we show that the DY cohomology does not change when moving
from $F$ to one of these upgraded
versions. In other words, DY cohomology is invariant under the adjunction of
colimits. We denote the diagrams indexing these colimits by boldface greek
letters~$\bGamma, \bDelta, \dots$.

In this section, we study the DY cohomology of a linear semigroupal functor $F:
\cC \to \cD$ which arises from a functor $F_0: \cC_0 \to \cD_0$ by such an
adjunction of colimits. More precisely, we consider in the following setting:

\begin{setup} \label{st:setup}
    Throughout the section, we work with a commutative diagram of
    $k$-linear semigroupal categories
    \[\begin{tikzcd}
        \cC_0 \arrow{r}{F_0} \arrow[hook]{d}{i_\cC} \arrow[bend right=50,swap]{dd}{Y_{\cC_0}} &[1.3cm] \cD_0 \arrow[hook,swap]{d}{i_\cD} \arrow[bend left=50]{dd}{Y_{\cD_0}} \\
        \cC \arrow{r}{F} \arrow[hook]{d}{j_\cC} & \cD \arrow[hook,swap]{d}{j_\cD} \\
        \PSh(\cC_0) \arrow{r}{\PSh(F_0)} & \PSh(\cD_0),
    \end{tikzcd}\]
    where $\PSh(\cC_0)$ is the category of $k$-linear presheaves $\cC_0^\op \to
    \CVec_k$. The functors $i_\cC, j_\cC, i_\cD, j_\cD$ are fully faithful and
    $Y$ denotes the Yoneda embedding. These embeddings will sometimes be left
    implicit.
\end{setup}

The semigroupal structure on $\PSh(\cC_0)$ and on the functors $Y$ and
$\PSh(F_0)$ is explained below. By \Cref{ss:functoriality}, \Cref{st:setup}
induces a commutative triangle of complexes
\begin{equation} \label{e:setup-triangle} \tag{\text{$\nabla$}}
\begin{tikzcd}
    C_\DY^\bullet(\PSh(F_0)) \arrow{rr} \arrow{rd} && C_\DY^\bullet(F)
    \arrow{ld} \\
    & C_\DY^\bullet(F_0).
\end{tikzcd}
\end{equation}

The main goal of this section is to prove the following

\begin{theorem} \label{t:colimit-invariant}
    The morphisms in the triangle~(\ref{e:setup-triangle}) are isomorphisms.
\end{theorem}

In particular, this is true whenever $F$ arises from $F_0$ as an additive or
pseudo-abelian envelope, $\Ind$-completion, etc.

\begin{definition}
    Let $G: \cA \to \cB$ be an arbitrary functor and $\cI$ a category. We say
    that $G$ \textbf{commutes with limits of shape $\cI$} if for all diagrams
    $\bDelta: \cI \to \cA$ for which the limit $\lim_{i \in \cI} \bDelta(i)$
    exists in $\cC$, the limit $\lim_{i \in \cI} G(\bDelta(i))$ exists in $\cB$
    and the natural morphism
    \[
        G\left(\lim_{i \in \cI} \bDelta(i)\right) \longrightarrow \lim_{i \in \cI}
        G(\bDelta(i))
    \]
    is an isomorphism. An analogous definition is made for colimits.
\end{definition}

\subsection{Categories of presheaves} \label{ss:presheaves}

In this subsection, we recall basic facts about the category of presheaves
$\PSh(\cC)$, which realizes the free cocompletion of $\cC$. We also consider an
alternative realization $\hat{\cC}$. We will later equip both constructions
with a semigroupal structure and extend cochains $\eta \in \CDY^n(\cC)$ to
cochains over the free cocompletion. To this end, $\hat{\cC}$ will give us more
control than $\PSh(\cC)$.

In the following, $\cV$ will be a complete and cocomplete closed symmetric
monoidal category, i.e.~$\cV$ has all small limits and colimits and the
functors $X \otimes -$ and $- \otimes X$ for $X \in \cV$ have right adjoints
(the internal hom) and therefore commute with colimits. We will be particularly
interested in the case where $\cV = \CVec_k$ is the category of vector spaces
over the field $k$. In addition, we fix a small $\cV$-enriched category $\cC$.
A comprehensive reference for enriched category theory is~\cite{kelly}.

The following is standard terminology: A $\cV$-enriched functor $\cC^\op \to
\cV$ is called a \textit{presheaf} on $\cC$. The $\cV$-enriched category of
presheaves and $\cV$-enriched natural transformations between them is denoted
by $\PSh_\cV(\cC)$. The functor
\[
    Y: \cC \to \PSh_\cV(\cC), \qquad X \mapsto \Hom_\cC(-, X)
\]
is called the \textit{Yoneda embedding}. Presheaves in the image of $Y$ are
called \textit{representable}.

\begin{proposition}[{Enriched Yoneda lemma, see \cite[\S~3.10]{kelly}}]
    \label{p:enriched-yoneda}
    For a presheaf $P: \cC^\op \to \cV$, there is an isomorphism
    \begin{equation} \label{e:coyoneda}
        \Hom_{\PSh_\cV(\cC)}(Y(X), P) \cong P(X) \cong \int^{A \in \cC} Y(A)(X)
        \otimes P(A)
    \end{equation}
    in $\cV$ natural in $X \in \cC$.
\end{proposition}

\begin{proposition}[{Coends as coequalizers see \cite[\S~2.1]{kelly}}]
    \label{p:coend}
    Let $F: \cC^\op \times \cC \to \cV$ be a $\cV$-enriched functor. There is a
    coequalizer diagram
    \[\begin{tikzcd}\displaystyle
        \bigsqcup_{A,B \in \cC} \Hom_\cC(B, A) \otimes F(A, B)
        \arrow[shift left=.75ex]{r}\arrow[shift right=.75ex]{r} & \displaystyle
        \bigsqcup_{A \in \cC} F(A, A)
        \arrow{r} & \displaystyle \int^{A \in \cC} F(A, A)
    \end{tikzcd}\]
\end{proposition}

\begin{proposition}[Pointwise colimits, see {\cite[\S~3.3]{kelly}}]
    \label{p:pointwise}
    Let $F: \cI \to \PSh_\cV(\cC)$ be a small diagram of presheaves. For $X \in
    \cC$, we have an isomorphism
    \[
        \left(\colim_{i \in \cI} F(i)\right)(X) \cong \colim_{i \in \cI}
        F(i)(X).
    \]
    in $\cV$, provided the colimit on the right exists.
\end{proposition}

\begin{corollary} \label{c:y-dense}
    For $\cV = \CVec_k$, every presheaf is isomorphic to a colimit of
    representable presheaves.
\end{corollary}
\begin{proof}
    Fix $P \in \PSh_\cV(\cC)$ and let $X \in \cC$. Combining
    \Cref{p:enriched-yoneda} and \Cref{p:coend}, we have a coequalizer diagram
    \begin{equation*}\begin{tikzcd} \displaystyle \label{e:psh-as-coeq}
        \bigsqcup_{A,B \in \cC}Y(A)(X) \otimes \Hom_\cC(A, B) \otimes P(A)
        \arrow[shift left=.75ex]{r}\arrow[shift right=.75ex]{r} & \displaystyle
        \bigsqcup_{A \in \cC} Y(A)(X) \otimes P(A) \arrow{r} & \displaystyle
        P(X).
    \end{tikzcd}\end{equation*}
    Moreover, for $A \in \cC$ and $V \in \CVec_k$, we can express $Y(A)(X)
    \otimes V$ as a coproduct $\bigsqcup_{i \in I} Y(A)(X)$, where $I$ is a basis
    of $V$. Thus $P(X)$ is expressed as the colimit of a
    diagram in $\CVec_k$ involving only $Y(A)(X)$ (for different $A \in \cC$),
    and the shape of this diagram depends only on $P$. Since colimits in
    $\PSh_\cV(\cC)$ are computed pointwise (\Cref{p:pointwise}), this realizes
    $P$ as a colimit of representable presheaves $Y(A)$.
\end{proof}

\begin{remark}
    The proof of \Cref{c:y-dense} depends crucially on the fact that we can
    choose a basis, or more generally that the monoidal unit in $\CVec_k$ is a
    generator. This implies that the coend expression~\eqref{e:coyoneda} can be
    turned into an (ordinary) colimit.

    In enriched category theory, the correct notion of a colimit is a
    \textit{weighted colimit} (see e.g.~\cite[Chapter~7]{Riehl_2014}), as
    opposed to an ordinary colimit (also called \textit{conical colimit} in this
    context). The coend~\eqref{e:coyoneda} is an example of such a
    weighted colimit: A presheaf $P$ is the colimit of the Yoneda embedding $Y:
    \cC \to \PSh_\cV(\cC)$ weighted by $P$, and in this sense, every presheaf
    is a (weighted) colimit of representable presheaves
    (see~\cite[Chapter~4]{Loregian_2021}). However, for categories $\cV$ where
    the unit object is a generator (such as $\cV = \CSet$ or $\cV = \CVec_k$),
    weighted limits can be expressed as conical limits as in \Cref{c:y-dense}.
\end{remark}

\Cref{c:y-dense} allows for a different description of $\PSh_\cV(\cC)$, namely
to identify diagrams in $\cC$ with their colimit in $\PSh_\cV(\cC)$. This
construction is known as the \textit{strict free cocompletion} (see
e.g.~\cite{BGP}) and reflects two common descriptions of the $\Ind$-completion:
either as the category of filtered diagrams $\cI \to \cC$ or as the subcategory
of $\PSh_{\CVec}(\cC)$ consisting of filtered colimits of representables.

\begin{definition}
    The \textbf{strict free cocompletion} $\hat{\cC}$ of $\cC$ is the
    $\cV$-enriched category defined as follows.

    \begin{itemize}
        \item The objects are the small diagrams in $\cC$, i.e.~functors
            $\bDelta: \cI \to \cC$ with $\cI$ small.
        \item Morphisms between small diagrams $\bGamma: \cI \to \cC, \bDelta:
            \cJ \to \cC$ are defined using the lim-colim-formula
            \begin{equation} \label{e:lim-colim}
                \Hom_{\hat{\cC}}(\bGamma, \bDelta) \coloneqq \lim_i \colim_j
                \Hom_{\cC}(\bGamma(i), \bDelta(j))
            \end{equation}
            with limits and colimits taken in $\cV$.
        \item Composition is inherited from the composition in $\cC$.
    \end{itemize}
\end{definition}

\begin{remark} \label{r:lim-colim}
    Let us make the lim-colim-formula \eqref{e:lim-colim} more explicit for
    $\cV = \CVec_k$. Consider small diagrams $\bGamma: \cI \to \cC$ and $\bDelta: \cJ \to
    \cC$ and fix $i \in \cI$. Then
    \[
        \colim_j \Hom_\cC(\bGamma(i), \bDelta(j)) = \frac{\bigoplus_j \Hom_\cC(\bGamma(i),
        \bDelta(j))}{R}
    \]
    where $R$ is the subspace spanned by $\vfi - \bDelta(g)\psi$ for all morphisms
    $g: j_1 \to j_2$ in $\cJ$ and all $\vfi: \bGamma(i) \to \bDelta(j_2), \psi: \bGamma(i) \to
    \bDelta(j_1)$ in $\cC$. Finally, we let $i$ vary and obtain that the limit
    \[
        \lim_j \colim_j \Hom_\cC(\bGamma(i), \bDelta(j)) \subseteq \prod_i \colim_j
        \Hom_\cC(\bGamma(i), \bDelta(j))
    \]
    consists of collections $\left(\vfi_i: \bGamma(i) \to \bDelta(j_i)\right)_{i \in \cI}$
    such that for all morphisms $f: i_1 \to i_2$ in $\cI$, $\vfi_{i_1} =
    \vfi_{i_2}\bGamma(f)$ in the colimit.
\end{remark}

\begin{remark}
    For $\cI = \{*\}$, we can always realize $\cC$ as the full subcategory of
    $\hat{\cC}$ consisting of diagrams of shape $\cI$.
\end{remark}

Later, we want $\hat{\cC}$ and $\PSh_\cV(\cC)$ to be equivalent. This works
because the lim-colim-formula also holds in $\PSh_\cV(\cC)$:

\begin{proposition} \label{p:ff-psh}
    There is a fully faithful $\cV$-enriched functor $L: \hat{\cC} \to
    \PSh_\cV(\cC)$ sending a diagram $\bDelta: \cI \to \cC$ to its colimit
    $\colim_i Y(\bDelta(i))$.
\end{proposition}
\begin{proof}
    Let $\bGamma: \cI \to \cC$ and $\bDelta: \cJ \to \cC$ be small diagrams. Then
    \[\begin{aligned}
        &\;\, \Hom_{\PSh_\cV(\cD)}\left(\colim_i Y(\bGamma(i)), \colim_j Y(\bDelta(j))\right) \\
        \cong &\; \lim_i \Hom_{\PSh(\cV(\cC))}(Y(\bGamma(i)), \colim_j Y(\bDelta(j))) \\
        \cong &\; \lim_i \left(\colim_j Y(\bDelta(j))\right)(\bGamma(i)) \\
        \cong &\; \lim_i \colim_j Y(\bDelta(j)(\bGamma(i)) \\
        =&\; \lim_i \colim_j \Hom_\cC(\bGamma(i), (\bDelta(j)).
    \end{aligned}\]
    The first isomorphism is due to the definition of limits and colimits; the
    second is the Yoneda lemma. The third isomorphism uses that colimits of
    presheaves are computed pointwise (\Cref{p:pointwise}). Thus, the
    assignment $L: \bDelta \mapsto \colim_i Y(\bDelta(i))$ extends to a fully
    faithful functor $L: \hat{\cC} \to \PSh_\cV(\cC)$.
\end{proof}

With the following corollary, we have reached the milestone of finding an
alternative description of $\PSh_\cV(\cC)$ using only diagrams. For $\cV =
\CSet$, this appears as \cite[Proposition~3.13]{BGP}.

\begin{corollary} \label{c:equiv-psh}
    For $\cV = \CVec_k$, the functor $L: \hat{\cC} \to \PSh_{\cV}(\cC)$ is an
    equivalence of categories.
\end{corollary}
\begin{proof}
    By \Cref{p:ff-psh}, $L$ is fully faithful. It is essentially surjective by
    \Cref{c:y-dense}.
\end{proof}

\subsection{Presheaves on semigroupal categories} \label{ss:finish}

The purpose of this subsection is to establish an equivalence $\hat{\cC} \cong
\PSh(\cC)$ as \textit{semigroupal} categories. This will imply that the
Davydov-Yetter complex is invariant under the free cocompletion. We only
consider categories enriched over $\cV = \CVec_k$, i.e.~$k$-linear categories,
where $k$ is a field. Let $\cC$ be a small strict $k$-linear semigroupal category
and abbreviate $\PSh(\cC) \coloneqq \PSh_{\CVec_k}(\cC)$.

The following two statements arise from well-known results about the Day
convolution for monoidal categories by removing the requirement of unitality,
as is done in~\cite[Theorem~2.24]{matti}:

\begin{proposition}[see {\cite{DayThesis}}] \label{p:day-cocont}
    There is a closed $k$-linear semigroupal structure (the \textit{Day
    convolution} $\oday$) on $\PSh(\cC)$. In particular, $\oday$ commutes
    with colimits in both variables. Moreover, the Yoneda embedding $Y: \cC \to
    \PSh(\cC)$ can be equipped with the structure $(Y, \Phi)$ of a
    $k$-linear semigroupal functor.
\end{proposition}

\begin{proposition}[see {\cite[Theorem~5.1]{kellyday}}]
    \label{p:free-semigroupal-cocompletion}
    Let $\cD$ be a cocomplete $k$-linear semigroupal category whose semigroupal
    product commutes with colimits in both variables. Then precomposition with
    the Yoneda embedding $Y: \cC \to \PSh(\cC)$ induces an equivalence
    \[
        \Fun_\otimes^{\text{cocont}}(\PSh(\cC), \cD) \longrightarrow
        \Fun_\otimes(\cC, \cD)
    \]
    between the category of cocontinuous $k$-linear semigroupal functors $\PSh(\cC)
    \to \cD$ and the category of $k$-linear semigroupal functors $\cC \to \cD$.
    Consequently, a $k$-linear semigroupal functor $F: \cC \to \cD$ induces
    a cocomplete $k$-linear semigroupal functor $\PSh(F): \PSh(\cC) \to \PSh(\cD)$.
\end{proposition}

\begin{lemma} \label{l:preserves-relevant}
    Assume the situation of~\Cref{st:setup}.

    \begin{enumerate}
        \item Every $A \in \cC$ appears as a colimit of some diagram
            $\bDelta_A: \cI_A \to \cC_0$.
        \item For every $X \in \cC$, the functors $(X \otimes -)$, $(-
            \otimes X)$, and $F$ commute with colimits of diagrams~$\bDelta_A$.
    \end{enumerate}
\end{lemma}
\begin{proof}
    This is because inclusions of full subcategories reflect colimits. For
    item~(1), let $A \in \cC$. By~\Cref{c:y-dense}, there exists a diagram
    $\bDelta_A: \cI \to \cC_0$ such that
    \[
        j_\cC(A) \cong \colim Y\bDelta_A = \colim j_\cC i_\cC \bDelta_A.
    \]
    Since $j_\cC$ reflects colimits, we have $A \cong \colim i_\cC \bDelta_A$.

    For~(2), let $A \in \cC$. Because $\PSh(F_0)$ is cocontinuous, we have
    \[\begin{aligned}
        j_\cD F(A) &= \PSh(F_0) j_\cC(A) \cong \PSh(F_0)\left(\colim j_\cC i_\cC
        \bDelta_A\right) \\
        &\cong \colim \PSh(F_0) Y \Delta = \colim j_\cD F i_\cC \bDelta_A.
    \end{aligned}\]
    Since $j_\cD$ reflects colimits, we have $F(A) \cong \colim F i_\cC
    \bDelta_A$. Because the Day convolution is cocontinuous in both variables,
    an analogous argument shows that the functors $(X \otimes -)$ and $(-
    \otimes X)$ for $X \in \cC$ commute with colimits of diagrams~$\bDelta_A$.
\end{proof}

Recall that~\Cref{st:setup} induces the commutative triangle
\begin{equation} \tag{\ref{e:setup-triangle}}
\begin{tikzcd}
    C_\DY^\bullet(\PSh(F_0)) \arrow{rr} \arrow{rd} && C_\DY^\bullet(F)
    \arrow{ld} \\
    & C_\DY^\bullet(F_0).
\end{tikzcd}
\end{equation}

\begin{proposition} \label{p:injective}
    All morphisms in the triangle~(\ref{e:setup-triangle}) are injective.
\end{proposition}
\begin{proof}
    It suffices to prove that the morphism $C_\DY^\bullet(F) \to C_\DY^\bullet(F_0)$ is
    injective, since $C_\DY^\bullet(\PSh(F_0)) \to C_\DY^\bullet(F_0)$ is only
    a special case of this morphism and this in turn implies injectivity of the
    third morphism. Let $\eta \in C_\DY^n(F)$, where we assume for simplicity that
    $n = 2$. Let $A, B \in \cC$. By~\Cref{l:preserves-relevant}, we can write
    \[
        F(A \otimes B) = \colim_{i \in \cI_A, j \in \cI_B} F(\bDelta_A(i)
        \otimes \bDelta_B(j))
    \]
    for some diagrams $\bDelta_A: \cI_A \to \cC_0, \bDelta_B: \cI_B \to \cC_B$
    and the colimit comes with natural morphisms $F(\bDelta_A(i) \otimes
    \bDelta_B(j)) \to F(A \otimes B)$ for each $i \in \cI_A, j \in \cI_B$. By
    naturality of $\eta$, we have commutative squares
    \[\begin{tikzcd}[column sep=2cm]
        F(\bDelta_A(i) \otimes \bDelta_B(j))
        \arrow{r}{\eta_{\bDelta_A(i),\bDelta_B(j)}} \arrow{d} &
        F(\bDelta_A(i) \otimes \bDelta_B(j)) \arrow{d} \\
        F(A \otimes B) \arrow{r}{\eta_{A,B}} & F(A \otimes B)
    \end{tikzcd}\]
    and the universal property of the colimit implies that $\eta_{A,B}$ is
    uniquely determined by the $\eta_{\bDelta_A(i),\bDelta_B(j)}$. More
    generally, every cochain in $C_\DY^n(F)$ for some $n$ is determined by its
    values on objects in $\cC_0$.
\end{proof}

\subsection{Completion of the proof} In order to prove surjectivity
in~\eqref{e:setup-triangle}, we work with the strict version of the
cocompletion.

\begin{proposition}
    The category $\hat{\cC}$ can be equipped with a strict $k$-linear
    semigroupal structure extending the structure on $\cC \subseteq \hat{\cC}$.
\end{proposition}
\begin{proof}
    We define the semigroupal product of diagrams $\bGamma: \cI \to \cC, \bDelta: \cJ \to \cC$
    to be the diagram
    \[
        \bGamma \ohat \bDelta:  \cI \times \cJ \xrightarrow{\bGamma \times \bDelta} \cC \times
        \cC \xrightarrow{\otimes} \cC.
    \]
    The semigroupal product on $\cC$ then canonically extends to $\hat{\cC}$ on
    morphisms using the lim-colim-formula~\eqref{e:lim-colim}.
\end{proof}

\begin{proposition} \label{p:equiv-semigroupal}
    The functor $L: \hat{\cC} \to \PSh(\cC)$ from \Cref{p:ff-psh}
    can be endowed with a $k$-linear semigroupal structure, making it an
    equivalence of $k$-linear semigroupal categories.
\end{proposition}
\begin{proof}
    For diagrams $\bGamma: \cI \to \cC, \bDelta: \cJ \to \cC$, consider the
    natural isomorphism
    \[\begin{aligned}
        \Psi_{\bGamma,\bDelta}: L(\bGamma \ohat \bDelta) &= \colim_{i,j} Y(\bGamma(i) \otimes \bDelta(j)) \\
        &\cong \colim_{i,j} Y(\bGamma(i)) \oday Y(\bDelta(j)) \\
        &\cong \colim_i \left(\bGamma(i) \oday \colim_j \bDelta(j)\right) \\
        &\cong \left(\colim_i \bGamma(i)\right) \oday \left(\colim_j \bDelta(j)\right) =
        L(\bGamma) \oday L(\bDelta)
    \end{aligned}\]
    where the first isomorphism comes from the semigroupal structure
    of the Yoneda embedding and the other isomorphisms use the fact that $\oday$
    is cocontinuous in both variables. We omit checking the hexagon axiom for
    $\Psi$.
\end{proof}

Let $F: \cC \to \cD$ be a strict $k$-linear semigroupal functor
between strict semigroupal categories. In addition to $\PSh(F)$, it induces a
functor $\hat{F}: \hat{\cC} \to \hat{\cD}$ sending a diagram $\bDelta: \cI \to
\cC$ to the diagram $F \circ \bDelta: \cI \to \cD$. On morphisms, $\hat{F}$
extends $F$ via the lim-colim-formula \eqref{e:lim-colim}. By construction,
$\hat{F}$ is $k$-linear strict semigroupal and we have the following diagram:
\[\begin{tikzcd}
    & \hat{\cC} \arrow{r}{\hat{F}} \arrow[dd,"L","\cong"'] &[2cm] \hat{\cD} \arrow[dd,"L","\cong"'] \\
    \cC \arrow[hook]{ru} \arrow[swap]{rd}{Y} && \\
    & \PSh(\cC) \arrow[swap]{r}{\PSh(\cD)} & \PSh(\cD)
\end{tikzcd}\]
All paths from $\cC$ to $\PSh(\cD)$ are isomorphic semigroupal functors. By
\Cref{p:free-semigroupal-cocompletion}, $\PSh(F)$ and $L \hat{F} L^{-1}$ are
isomorphic as functors $\PSh(\cC) \to \PSh(\cD)$. In particular, $\hat{F}$ is
also cocontinuous. In the commutative diagram
\[\begin{tikzcd}
    \cC \arrow{r}{F} \arrow[hook]{d} \arrow[bend right=50,swap]{dd}{Y} &[1.3cm] \cD \arrow[hook]{d} \arrow[bend left=50]{dd}{Y} \\
    \hat{\cC} \arrow{r}{\hat{F}} \arrow{d}{L} & \hat{\cD} \arrow[swap]{d}{L} \\
    \PSh(\cC) \arrow{r}{\PSh(F)} & \PSh(\cD),
\end{tikzcd}\]
the three arrows on the right are fully faithful. By
\Cref{p:functoriality-square}, this induces a commutative triangle
\begin{equation} \label{e:triangle-psh-hat}
\begin{tikzcd}
    C_\DY^\bullet(\PSh(F)) \arrow{rr}{\cong} \arrow[hook]{rd} &&
    C_\DY^\bullet(\hat{F}) \arrow{ld}{\vfi} \\
    & C_\DY^\bullet(F),
\end{tikzcd}
\end{equation}
where the top arrow is an isomorphim, being induced by the equivalence $L$, and
the left arrow is injective by~\Cref{p:injective}.

\begin{lemma}
    For a $k$-linear strict semigroupal functor $F: \cC \to \cD$, the forgetful
    morphism $\vfi: C_\DY^\bullet(\hat{F}) \to C_\DY^\bullet(F)$
    in~\eqref{e:triangle-psh-hat} is an isomorphism.
\end{lemma}
\begin{proof}
    It suffices to prove surjectivity. Given $\eta \in C^n_\DY(\cC)$, we want
    to extend it to an element of $\hat{\eta} \in C^n_\DY(\hat{\cC})$ using the
    description of morphisms in $\hat{\cC}$ given in \Cref{r:lim-colim}. To
    minimize notational overhead, we only prove the case $n = 2$. It will then
    be straightforward to generalize. For diagrams $\bGamma: \cI \to \cC,
    \bDelta: \cJ \to \cC$ and $i \in \cI, j \in \cJ$, we have a morphism
    \[
        \eta_{\bGamma(i),\bDelta(j)}: F(\bGamma(i) \otimes \bDelta(j))
        \longrightarrow F(\bGamma(i) \otimes \bDelta(j))
    \]
    in $\cD$. Define
    \[
        \hat{\eta}_{\bGamma,\bDelta} \coloneqq
        (\eta_{\bGamma(i),\bDelta(j)})_{i\in\cI,j\in\cJ}\quad:\quad
        \hat{F}(\bGamma \ohat \bDelta) \longrightarrow \hat{F}(\bGamma \ohat
        \bDelta)
    \]
    To see that this is a well-defined morphism in $\hat{\cD}$, let $f: i_1
    \to i_2, g: j_1 \to j_2$ be morphisms in $\cI$ and $\cJ$, respectively.
    Since $\eta \in C^2_\DY(\cC)$, the diagram
    \[\begin{tikzcd}[column sep=2.5cm]
        F(\bGamma(i_1) \otimes \bDelta(j_1)) \arrow{r}{F(\bGamma(f) \otimes \bDelta(g))}
        \arrow[swap]{d}{\eta_{\bGamma(i_1),\bDelta(j_1)}} & F(\bGamma(i_2) \otimes \bDelta(j_2))
        \arrow{d}{\eta_{\bGamma(i_2),\bDelta(j_2)}} \\
        F(\bGamma(i_1) \otimes \bDelta(j_1)) \arrow[swap]{r}{F(\bGamma(f) \otimes \bDelta(g))} & F(\bGamma(i_2)
        \otimes \bDelta(j_2))
    \end{tikzcd}\]
    commutes, so $\eta_{\bGamma(i_1),\bDelta(j_1)}$ is equivalent to $\eta_{\bGamma(i_2),\bDelta(j_2)}
    \,F(\bGamma(f) \otimes \bDelta(g))$ in the sense of \Cref{r:lim-colim}, as
    desired.

    To see that $\hat{\eta}_{\bGamma,\bDelta}$ is natural in $\bGamma$, consider a morphism
    $\vfi: \bGamma \to \bGamma'$ in $\hat{\cC}$, where $\bGamma': \cI' \to \cC$ is some
    diagram. We need the following diagram to commute:
    \[\begin{tikzcd}[column sep=2cm]
        \hat{F}(\bGamma \ohat \bDelta) \arrow{r}{\vfi \ohat \id_\bDelta}
        \arrow[swap]{d}{\hat{\eta}_{\bGamma,\bDelta}} & \hat{F}(\bGamma' \ohat \bDelta)
        \arrow{d}{\hat{\eta}_{\bGamma',\bDelta}} \\
        \hat{F}(\bGamma \ohat \bDelta) \arrow[swap]{r}{\vfi \ohat \id_\bDelta} &
        \hat{F}(\bGamma' \ohat \bDelta)
    \end{tikzcd}\]
    This can be checked component-wise, so let $(i, j) \in \cI \times \cJ$: The
    component of $\vfi$ at $i$ is some morphism $\vfi_i: \bGamma(i) \to \bGamma'(i')$, and
    the diagram
    \[\begin{tikzcd}[column sep=2cm,row sep=1.0cm]
        F(\bGamma(i) \otimes \bDelta(j)) \arrow{r}{\vfi_i \otimes \id_{\bDelta(j)}}
        \arrow[swap]{d}{\eta_{\bGamma(i),\bDelta(j)}} & F(\bGamma'(i') \otimes
        \bDelta(j))
        \arrow{d}{\eta_{\bGamma'(i'),\bDelta(j)}} \\
        F(\bGamma(i) \otimes \bDelta(j)) \arrow[swap]{r}{\vfi_i \otimes \id_{\bDelta(j)}}
        & F(\bGamma'(i') \otimes \bDelta(j))
    \end{tikzcd}\]
    commutes.
\end{proof}

\begin{corollary} \label{c:iso-forget-psh}
    For a $k$-linear strict semigroupal functor $F: \cC \to \cD$, the forgetful
    morphism of complexes $C_\DY^\bullet(\PSh(F)) \to C_\DY^\bullet(F)$ is an
    isomorphism.
\end{corollary}

We can now prove the main theorem of this section.

\begin{proof}[{Proof of \Cref{t:colimit-invariant}}]
    We may assume that $F$ is strict because strictification is functorial.
    Recall that we have the triangle
    \begin{equation} \tag{\ref{e:setup-triangle}}
    \begin{tikzcd}
        C_\DY^\bullet(\PSh(F_0)) \arrow{rr} \arrow{rd} && C_\DY^\bullet(F)
        \arrow{ld} \\
        & C_\DY^\bullet(F_0).
    \end{tikzcd}
    \end{equation}
    By \Cref{p:injective}, each arrow is a monomorphism. By
    \Cref{c:iso-forget-psh}, the left arrow is an isomorphism, so all three are
    isomorphisms.
\end{proof}

\subsection{Applications} \label{ss:applications}

Let us collect a few consequences of~\Cref{t:colimit-invariant}.

\begin{example}
    The functor $F$ in \Cref{st:setup} may be understood to be a completion
    of $F_0$ with respect to certain types of (co)limits: A $k$-linear semigroupal
    functor $F: \cC \to \cD$ extends canonically to the additive completion
    $F^\oplus: \cC^\oplus \to \cD^\oplus$, the pseudo-abelian envelope $F^\ps:
    \cC^\ps \to \cD^\ps$ and the Ind-completion $\Ind(F): \Ind(\cC) \to
    \Ind(\cD)$. By \Cref{t:colimit-invariant}, the Davydov-Yetter cohomology of these
    functors coincides with $C_\DY^\bullet(F)$.
\end{example}

An important special case is $F = \id_\cC$:

\begin{corollary} \label{c:sandwich}
    Consider $k$-linear semigroupal subcategories $\cC_0 \subseteq \cC \subseteq
    \PSh(\cC_0)$. The inclusion $\cC_0 \subseteq \cC$ induces an isomorphism
    $C_\DY^\bullet(\cC) \to C_\DY^\bullet(\cC_0)$.
\end{corollary}

\begin{example} \label{x:enough-projectives}
    Let $\cC$ be an abelian $k$-linear semigroupal category with enough
    projectives. Let $\cP$ be the full subcategory of projective objects in
    $\cC$, which is also $k$-linear semigroupal. Every object $X \in \cC$
    has a projective resolution
    \[
        \cdots \to P_1 \to P_0 \to X \to 0
    \]
    and thus appears as the cokernel of $P_1 \to P_0$, i.e.~as a colimit of
    objects in $\cP$ in a canonical way. Consequently, the Yoneda embedding
    $\cP \hookrightarrow \PSh(\cP)$ factors as $\cP \subseteq \cC \subseteq
    \PSh(\cP)$, and \Cref{c:sandwich} yields an isomorphism of complexes
    $C_\DY^\bullet(\cC) \to C_\DY^\bullet(\cP)$. In other words, in order to
    compute the DY cohomology of $\cC$, it suffices to consider only the
    projective objects.
\end{example}

 Here is the final application of \Cref{t:colimit-invariant} presented here.
 For pseudo-tensor categories, a universal embedding into a tensor category
 (which is abelian by definition) need not exist. If it does, we speak of the
 \textit{abelian envelope}. We show that in many cases, passing to the abelian
 envelope does not change the Davydov-Yetter cohomology.

\begin{definition}[see {\cite{EAHS}}]
    Let $\cC$ be a pseudo-tensor category over a field $k$. A $k$-linear
    monoidal functor $F: \cC \to \cA$, where $\cA$ is a $k$-linear tensor
    category, is called a \textbf{(monoidal) abelian envelope} of $\cC$ if it
    satisfies the following universal property: For any $k$-linear tensor
    category $\cT$, precomposition with $F$ defines an equivalence
    \[
        \Fun_\otimes^\ex(\cA, \cT) \to \Fun_\otimes^\faith(\cC, \cT), \qquad G
        \mapsto G \circ F
    \]
    between the category $\Fun_\otimes^\ex(\cA, \cT)$ of exact $k$-linear
    monoidal functors and the category $\Fun_\otimes^\faith(\cC, \cT)$ of
    faithful $k$-linear monoidal functors.
\end{definition}

For a field $k$ and a pseudo-tensor category $\cC$, Coulembier
\cite{Cou} gives an existence criterion for an abelian envelope of $\cC$. He
defines a full $k$-linear monoidal subcategory $\Sh(\cC) \subseteq \PSh(\cC)$
of \textit{sheaves} on $\cC$ and the notion of a \textit{strongly faithful
object} in $\cC$ (for precise definitions, see \cite[Definition 2.2.2 and
\S~3.1.2]{Cou}).

\begin{proposition}[Existence Theorem] \label{p:abenv-existence}
    If every morphism in $\cC$ is split by some strongly faithful object $X$,
    then $\cC$ admits an abelian envelope $\cC \hookrightarrow \cA$. In this
    case we have a commutative (up to isomorphism) diagram of pseudo-tensor
    categories
    \[ \begin{tikzcd}
        \Sh(\cC) \arrow{r}{\cong} & \Ind(\cA) \\
        \cC \arrow[hook]{u} \arrow[hook]{r} & \cA. \arrow[hook]{u}
    \end{tikzcd}\]
\end{proposition}
\begin{proof}
    This is \cite[Theorem~4.1.1]{Cou}.
\end{proof}

The existence theorem provides enough information about the abelian envelope to
apply the invariance theorem and obtain an invariance result for abelian
envelopes.

\begin{theorem} \label{t:abelian-envelope}
    In the setting of the Existence Theorem~(\Cref{p:abenv-existence}), we
    have an isomorphism $C^\bullet_\DY(\cA) \cong C^\bullet_\DY(\cC)$ induced
    by the embedding $\cC \hookrightarrow \cA$.
\end{theorem}
\begin{proof}
    By the Existence theorem, we have an induced commutative square
    \[\begin{tikzcd}
        C_\DY^\bullet(\Sh(\cC)) \arrow{d} & C_\DY^\bullet(\Ind(\cA)) \arrow[swap]{l}{\cong} \arrow{d} \\
        C_\DY^\bullet(\cC)  & C_\DY^\bullet(\cA). \arrow{l}
    \end{tikzcd}\]
    Applying \Cref{c:sandwich} to the chains of subcategories
    $\cC \subseteq \Sh(\cC) \subseteq \PSh(\cC)$ and $\cA \subseteq \Ind(\cA)
    \subseteq \PSh(\cA)$, we see that the left and right arrow are
    isomorphisms. Thus the bottom arrow is an isomorphism.
\end{proof}

%% file: bibliography.tex
\bibliographystyle{alphaurl}
\bibliography{refs}